\documentclass[12pt,reqno]{amsart}
\usepackage{palatino} 
\usepackage{amsmath,amssymb,amsxtra,color,calligra,mathrsfs}

\usepackage[colorlinks]{hyperref}
\hypersetup{linkcolor=blue,citecolor=blue,filecolor=dullmagenta,urlcolor=blue}
\usepackage[all]{xy}
\usepackage{tikz,tikz-cd,tkz-graph,enumerate}
\usetikzlibrary{matrix,arrows,decorations.pathmorphing}

\DeclareMathOperator{\sHom}{\mathcal{H}\!\textit{om}}

\DeclareMathOperator{\sEnd}{\mathcal{E}\!\textit{nd}}
\DeclareMathOperator{\gr}{\textnormal{gr}}

\DeclareMathOperator{\rk}{\mathrm{rk}}
\DeclareMathOperator{\Id}{\textnormal{Id}}

\DeclareMathOperator{\Aut}{\textnormal{Aut}}

\newtheorem{theorem}{Theorem}[section]
\newtheorem{lemma}[theorem]{Lemma}
\newtheorem{proposition}[theorem]{Proposition}
\newtheorem{corollary}[theorem]{Corollary}

\theoremstyle{definition}
\newtheorem{definition}{Definition}[section]
\newtheorem{remark}{Remark}[section]
\newtheorem{example}{Example}[section]

\numberwithin{equation}{section}

\begin{document}

\baselineskip=15.5pt

\title[Hodge Bundles and Generalized Opers]{System of Hodge Bundles and Generalized Opers on Smooth Projective Varieties} 

\author[S. Basu]{Suratno Basu} 
\address{The Institute of Mathematical Sciences, 4th Cross Street, CIT Campus, 
Taramani, Chennai 600113, Tamil Nadu, India.}
\email{suratnob@imsc.res.in}

\author[A. Paul]{Arjun Paul} 
\address{Department of Mathematics, Indian Institute of Technology Bombay, Powai, Mumbai 400076, Maharashtra, India.} 

\email{arjun.math.tifr@gmail.com} 

\author[A. Saha]{Arideep Saha} 
\address{The Institute of Mathematical Sciences, 4th Cross Street, CIT Campus, 
Taramani, Chennai 600113, Tamil Nadu, India.}
\email{arideep@imsc.res.in}



\begin{abstract}
	Let $k$ be an algebraically closed field of any characteristic. 
	Let $X$ be a polarized irreducible smooth projective algebraic variety over $k$. 
	We give criterion for semistability and stability of system of Hodge bundles on $X$. 
	We define notion of generalized opers on $X$, and prove semistability of the Higgs bundle 
	associated to generalized opers. We also show that existence of partial oper structure on a 
	vector bundle $E$ together with a connection $\nabla$ over $X$ implies semistability of the pair 
	$(E, \nabla)$. 
\end{abstract}

\maketitle

\section{Introduction}
Let $k$ be an algebraically closed field of characteristic $char(k) \geq 0$. Let $X$ be a polarized irreducible 
smooth projective algebraic variety over $k$. Let $E$ be a vector bundle on $X$ together with a filtration 
\begin{equation}\label{eqn-1.1}
\mathcal{F}^{\bullet}(E)\, : \,\,\, 
E = \mathcal{F}^0(E) \supsetneq \mathcal{F}^1(E) \supsetneq \cdots \supsetneq \mathcal{F}^{n-1}(E) 
\supsetneq \mathcal{F}^n(E) = 0\, 
\end{equation}
where $\mathcal{F}^i(E)$ are subbundles of $E$, for all $i$. Suppose that $E$ admits a flat algebraic 
connection $\nabla : E \to E \otimes \Omega_X^1$ such that the filtration \eqref{eqn-1.1} is 
Griffiths transversal with respect to $\nabla$; meaning that 
$\nabla(\mathcal{F}^i(E)) \subseteq \mathcal{F}^{i-1}(E) \otimes \Omega_X^1$, for all $i = 1, \ldots, n-1$. 
Then $\nabla$ induces a Higgs field $\theta_{\nabla}$ on the associated vector bundle 
$\gr(\mathcal{F}^{\bullet}(E)) := \bigoplus\limits_{i=0}^{n-1} \gr^i(\mathcal{F}^{\bullet}(E))$, where 
$\gr^i(\mathcal{F}^{\bullet}(E)) := \mathcal{F}^i(E)/\mathcal{F}^{i+1}(E)$, for all $i = 0, 1, \ldots, n-1$. 

In \cite{Si}, it is shown that given a flat connection $\nabla$ on $E$, there exists a Griffiths transversal 
filtration $\mathcal{F}^{\bullet}(E)$ such that the associated Higgs bundle 
$(\gr(\mathcal{F}^{\bullet}(E)), \theta_{\nabla})$ is semistable. However, given a Griffiths transversal 
filtration $\mathcal{F}^{\bullet}(E)$ of $E$ with respect to $\nabla$, it is not known, in general, 
whether the associated Higgs bundle $(\gr(\mathcal{F}^{\bullet}(E)), \theta_{\nabla})$ is semistable. 
When $X$ is a compact Riemann surface, in \cite[p.~156]{Bi}, the following possible solution to this 
problem is proposed: 

Let $E$ be a holomorphic vector bundle on $X$ whose all indecomposable components has degree zero. 
Let $\mathcal{F}^{\bullet}(E)$ be a filtration of $E$ by its subbundles on $X$. Then $E$ admits a holomorphic 
connection $\nabla$ such that $\mathcal{F}^{\bullet}(E)$ is Griffiths transversal with respect to $\nabla$ 
if and only if $\gr(\mathcal{F}^{\bullet}(E))$ admits a holomorphic Higgs field $\theta$ such that 
\begin{itemize}
	\item $(\gr(\mathcal{F}^{\bullet}(E)), \theta)$ is semistable, and 
	\item $\theta(\gr^i(\mathcal{F}^{\bullet}(E))) \subseteq 
	\gr^{i-1}(\mathcal{F}^{\bullet}(E)) \otimes \Omega_X^1$, for all $i$. 
\end{itemize}

Note that, the Higgs bundle 
$(\gr(\mathcal{F}^{\bullet}(E)), \theta_{\nabla})$ admits a structure of a system of Hodge bundles on $X$; 
meaning that, $\theta_{\nabla}(\gr^i(\mathcal{F}^{\bullet}(E))) \subseteq \gr^{i-1}(\mathcal{F}^{\bullet}(E)) 
\otimes \Omega_X^1$, for all $i$. Therefore, it is natural to ask, more generally, when 
does a Higgs bundle (not necessarily of degree zero) on $X$ having a structure of a system of Hodge 
bundles is semistable and when it is stable. We give a criterion for this. 

Fix an ample line bundle on $X$. We prove the following results : 
\begin{theorem}\label{thm-1.1}
	Assume that $char(k) \geq 0$, and $\Omega_X^1$ is semistable with $\deg(\Omega_X^1) \geq 0$. 
	Let $(E, \theta)$ be a Higgs bundle on $X$ having a structure of a system of Hodge bundles: 
	$E = \bigoplus\limits_{i=0}^n E_i$ such that $\theta\vert_{E_i} : E_i \stackrel{\simeq}{\longrightarrow} 
	E_{i-1} \otimes \Omega_X^1$ is an isomorphism, for all $i = 1, \ldots, n$. 
	Then $(E, \theta)$ is semistable if $E_i$ is semistable, for all $i = 0, 1, \ldots, n$. 
	The converse holds if $char(k) = 0$. 
\end{theorem}

\begin{theorem}
	Assume that $char(k) \geq 0$, and $\deg(\Omega_X^1) > 0$. 
	The Higgs bundle $(E, \theta)$ in Theorem \ref{thm-1.1} is stable if $E_i$ is stable, for all 
	$i = 0, 1, \ldots, n$. Converse holds if $\dim_k(X) = 1$. 
\end{theorem}

We give some examples to show that the isomorphism conditions 
$$\theta\vert_{E_i} : E_i \stackrel{\simeq}{\longrightarrow} E_{i-1} \otimes \Omega_X^1\,,\,\, 
\forall\,\, i = 1, \ldots, n\,$$ 
and semistability of $E_i$, for all $i = 0, 1, \ldots, n$, in the Theorem \ref{thm-1.1} are crucial 
for semistability of $(E, \theta)$. 

Finally, we give a criteria on the Griffiths transversal filtration for a flat connection, 
which we refer to as ``\textit{generalized oper}'' so that the associated Higgs bundle becomes semistable. 
We define notion of semistability of connections, and prove the following : 
\begin{theorem}\label{thm-1.3}
	Let $X$ be a polarized smooth projective variety over $k$, and let $\Omega_X^1$ be semistable of non-negative degree. 
	Let $E$ be a vector bundle on $X$ together with a connection (not necessarily flat) 
	$\nabla : E \longrightarrow E \otimes \Omega_X^1$. Let 
	$$\mathcal{F}^{\bullet}(E) : 0 = \mathcal{F}^n(E) \subsetneq \mathcal{F}^{n-1}(E) \subsetneq \cdots 
	\subsetneq \mathcal{F}^1(E) \subsetneq \mathcal{F}^0(E) = E$$ 
	be a $\nabla$-Griffiths transversal filtration of $E$ by its subbundles such that the induced $\mathcal{O}_X$--module 
	homomorphism 
	$\theta_{\nabla} : \gr(\mathcal{F}^{\bullet}(E)) \longrightarrow \gr(\mathcal{F}^{\bullet}(E)) \otimes \Omega_X^1$ 
	is a Higgs field on $\gr(\mathcal{F}^{\bullet}(E))$ (i.e., $\theta_{\nabla} \wedge \theta_{\nabla} = 0$ in 
	$H^0(X, \sEnd(\gr(\mathcal{F}^{\bullet}(E)))\otimes\Omega_X^2)$). If the Higgs bundle 
	$(\gr(\mathcal{F}^{\bullet}(E)), \theta_{\nabla})$ is semistable (respectively, stable), then the pair 
	$(E, \nabla)$ is semistable (respectively, stable). 
\end{theorem}
Theorem \ref{thm-1.3} is a sort of converse of \cite[Theorem 2.2]{LSYZ} by Lan-Sheng-Yang-Zuo. This also generalize 
\cite[Proposition 3.4.4]{JP} of Joshi-Pauly proved for the case of curve in positive characteristic.

\section{Griffiths Transversal Filtration}
\subsection{Preliminaries}
Let $k$ be an algebraically closed field of characteristic $char(k) \geq 0$. 
Let $X$ be an irreducible smooth projective algebraic variety over $k$. 
Let $\mathcal{O}_X$ be the sheaf of regular functions on $X$. 
Let $\Omega_X^1$ be the cotangent bundle of $X$. 
Let $E$ be a coherent sheaf of $\mathcal{O}_X$--modules on $X$. The rank of $E$ is defined 
to be the dimension of the generic fiber of $E$. We denote it by $\rk(E)$. 
Since $X$ is irreducible, this is well-defined. 
We say that $E$ is a \textit{vector bundle} on $X$ if it is locally free and of finite rank on $X$. 
An $\mathcal{O}_X$--submodule $F$ of a vector bundle $E$ is said to be a \textit{subbundle} of $E$ 
if $F$ is locally free and the quotient sheaf $E/F$ is torsion free on $X$. 

Let $E$ be a vector bundle on $X$. 
\begin{definition}\label{connection}
	A \textit{  connection} on $E$ is a $k$--linear sheaf homomorphism 
	\begin{equation}
		\nabla : E \longrightarrow E \otimes \Omega_X^1\,,
	\end{equation}
	satisfying the following Leibniz rule:
	\begin{equation}\label{leibniz-id}
		\nabla(f\cdot s)\, =\, s \otimes {\rm d}f\, +\, f\cdot \nabla(s)\,,
	\end{equation}
	for every section $s \in E(U)$ and regular function $f \in \mathcal{O}_X(U)$, 
	for any open subset $U \subset X$. 
\end{definition}

Let $\Omega_X^2 := \bigwedge^2\Omega_X^1$. 
Given a connection $\nabla$ on $E$, we can extend it to a $k$--linear sheaf homomorphism 
(denoted by the same symbol)
$$\nabla : E \otimes \Omega_X^1 \longrightarrow E \otimes \Omega_X^2$$ 
satisfying $\nabla(s \otimes \omega) = s \otimes d\omega - \nabla(s) \wedge \omega$, 
for all local sections $s \in E(U)$ and $\omega \in \Omega_X^1(U)$. 
This defines an element 
$$\kappa(\nabla) := \nabla \circ \nabla \in H^0(X, \sEnd(E)\otimes\Omega_X^2)\,,$$ 
called the \textit{curvature} of $\nabla$. 
A connection $\nabla$ is said to be \textit{flat} if $\kappa(\nabla) = 0$. 

\begin{definition}\label{G-transv}
	Let $E$ be a vector bundle on $X$ and let 
	\begin{equation}\label{filtration}
		\mathcal{F}^{\bullet}(E) \,\, : \,\, E = \mathcal{F}^0(E) \supsetneq \mathcal{F}^1(E) 
		\supsetneq \cdots \supsetneq \mathcal{F}^{n-1}(E) \supsetneq \mathcal{F}^n(E) = 0\,,
	\end{equation}
	be a filtration of $E$ by its subbundles. The filtration $\mathcal{F}^{\bullet}(E)$ is said to 
	be {\it Griffiths transversal} for a flat connection $\nabla$ on $E$ if it satisfies the 
	following conditions:
	\begin{equation}\label{GTF}
	\nabla(\mathcal{F}^i(E)) \subseteq \mathcal{F}^{i-1}(E) \otimes \Omega_X^1\,, \,\, 
	\forall \, i = 1, \ldots, n-1\,.
	\end{equation}
\end{definition}

\subsection{Associated Higgs Bundle}\label{as-higgs-bun}
\begin{definition}
	A \textit{Higgs sheaf} on $X$ is a pair $(E, \theta)$, where $E$ is a coherent sheaf of 
	$\mathcal{O}_X$--modules on $X$ and $\theta : E \longrightarrow E\otimes\Omega_X^1$ is an 
	$\mathcal{O}_X$--module homomorphism such that the following composite $\mathcal{O}_X$--module 
	homomorphism vanishes identically: 
	\begin{equation}
	\theta\wedge\theta : E \stackrel{\theta}{\longrightarrow} E\otimes\Omega_X^1 
	\stackrel{\theta\otimes\Id_{\Omega_X^1}}{\longrightarrow} 
	E\otimes\Omega_X^1\otimes\Omega_X^1 \stackrel{\Id_E\otimes(-\wedge -)}{\longrightarrow} 
	E\otimes\Omega_X^2\,. 
	\end{equation}
\end{definition}

Consider a triple $(E, \mathcal{F}^{\bullet}(E), \nabla)$, where $E$ is a vector bundle on $X$ 
together with a flat connection $\nabla$, and a filtration 
$\mathcal{F}^{\bullet}(E)$ on $E$, as in \eqref{filtration}, which is Griffiths transversal for $\nabla$ 
(see \eqref{GTF}). Then $\nabla$ induces an $\mathcal{O}_X$--linear homomorphism 
\begin{equation}\label{theta-i}
	\theta_{\nabla}^i : {\rm gr}^i(\mathcal{F}^{\bullet}(E)) \longrightarrow 
	{\rm gr}^{i-1}(\mathcal{F}^{\bullet}(E)) \otimes \Omega_X^1\,,
\end{equation}
where ${\rm gr}^i(\mathcal{F}^{\bullet}(E)) = \mathcal{F}^i(E)/\mathcal{F}^{i+1}(E)$, 
for all $0 \leq i \leq n-1$, and ${\rm gr}^{-1}(\mathcal{F}^{\bullet}(E)) := 0$; 
the $\mathcal{O}_X$--linearity of $\theta_{\nabla}^i$ 
follows from the Leibniz rule \eqref{leibniz-id}. 
Thus we have an $\mathcal{O}_X$--linear homomorphism 
\begin{equation}\label{theta-nab}
	\theta_{\nabla} : {\rm gr}(\mathcal{F}^{\bullet}(E)) \longrightarrow 
	{\rm gr}(\mathcal{F}^{\bullet}(E)) \otimes \Omega_X^1\,, 
\end{equation}
where 
\begin{equation}\label{gr-bun}
	{\rm gr}(\mathcal{F}^{\bullet}(E)) = \bigoplus\limits_{i=0}^{n-1}{\rm gr}^i(\mathcal{F}^{\bullet}(E))\,.
\end{equation}
Note that, the flatness of $\nabla$ ensures that $\theta_{\nabla}\wedge \theta_{\nabla} = 0$. 
Therefore, $(\gr(\mathcal{E}^\bullet), \theta_{\nabla})$ is a Higgs bundle over $X$. 
Note that the Higgs field $\theta_{\nabla}$ satisfies $\theta_{\nabla}^n = 0$, 
and hence is nilpotent in the graded $k$--algebra 
$\bigoplus\limits_{i=0}^n H^0\left(X\,,\, \sEnd(E)\otimes \left(\Omega_X^1\right)^{\otimes i}\right)$, 
where $\sEnd(E)$ is the sheaf of $\mathcal{O}_X$--module endomorphisms of $E$. 

A \textit{polarization} on $X$ is given by choice of an ample line bundle $L$ on it. 
Fix an ample line bundle $L$ on $X$. Let $E$ be a non-zero coherent sheaf of 
$\mathcal{O}_X$--modules on $X$. Then the degree of $E$ with respect to $L$ is defined by 
$$\deg(E) := c_1(\det(E)) \cdot [L]^{n-1}\,,$$
where $\det(E)$ is the \textit{determinant} line bundle of $E$. 
If $\rk(E) > 0$, the ratio $\mu(E) := \deg(E)/\rk(E)$ is called the \textit{slope} of $E$. 

\begin{definition}\label{sshb}
	A torsion free Higgs sheaf $(E, \theta)$ on $X$ is said to be \textit{semistable} 
	(respectively,  {\it stable}) if for any non-zero proper subsheaf $F \subset E$ with 
	$0 < \rk(F) < \rk(E)$ and $\theta(F) \subseteq F \otimes \Omega_X^1$, we have 
	$$ \mu(F) \leq \mu(E)\,\,\, (\text{respectively,}\,\, \mu(F) <  \mu(E))\,.$$ 
\end{definition}

\begin{remark}
	A torsion free coherent sheaf $E$ on $X$ can be considered as a Higgs sheaf $(E, \theta)$ 
	with zero Higgs field $\theta = 0$ on $E$. Then the above notion of semistability and 
	stability coincides with the corresponding notions for torsion free coherent sheaves. 
\end{remark}

\begin{definition}\label{Higgs-hom}
	Let $(E_1, \theta_1)$ and $(E_2, \theta_2)$ be two Higgs sheaves on $X$. 
	A \textit{Higgs homomorphism} from $(E_1, \theta_1)$ to $(E_2, \theta_2)$ is given 
	by an $\mathcal{O}_X$--module homomorphism $\varphi : E_1 \longrightarrow E_2$ such that 
	$\theta_2\circ \varphi = (\varphi\times\Id_{\Omega_X^1})\circ \theta_1$. 
\end{definition}

\begin{lemma}\label{lem-2}
	Let $(E, \theta)$ and $(F, \phi)$ be two Higgs bundles on $X$. 
	Let $\Theta = \theta\otimes\Id_F + \Id_E\otimes\phi$. If $(E\otimes F, \Theta)$ is semistable, 
	then both $(E, \theta)$ and $(F, \phi)$ are semistable. Converse holds if the characteristic of $k$ is zero. 
\end{lemma}

\begin{proof}
	Suppose that $(E\otimes F, \Theta)$ is semistable. 
	If $(E, \theta)$ were not semistable, then there is a maximal destabilizing Higgs subsheaf 
	$(E_0, \theta\vert_{E_0})$ of $(E, \theta)$ with $\mu(E_0) > \mu(E)$. Since the functor 
	$-\otimes F$ is left exact, $(E_0 \otimes F, \theta\vert_{E_0}\otimes\Id_F + \Id_{E_0}\otimes\phi)$ 
	is a destabilizing subsheaf of $(E\otimes F, \Theta)$, with 
	$\mu(E_0 \otimes F) = \mu(E_0) + \mu(F) > \mu(E) + \mu(F) = \mu(E\otimes F)$, contradicting Higgs 
	semistability of $(E\otimes F, \Theta)$. Therefore, both $(E, \theta)$ and $(F, \phi)$ are semistable. 
	For the converse part, see \cite[Corollary 3.8, p.~38]{Si2}. 
\end{proof}

\section{System of Hodge bundles and semistability}\label{sec-3}
Let $X$ be an irreducible smooth projective algebraic variety over $k$ together with a fixed 
ample line bundle on it. 

\begin{definition}
	A Higgs bundle $(E, \theta)$ is said to have a structure of a \textit{system of Hodge bundles} if 
	$E$ has a direct sum decomposition $E = \bigoplus\limits_{i=0}^n E_i$ by its subbundles $E_i$ such that 
	$\theta(E_i) \subseteq E_{i-1}\otimes\Omega_X^1$, for all $0 \leq i \leq n$, with $E_{-1} = 0$. 
\end{definition}

\subsection{Criterion for semistability of a system of Hodge bundles}
Now we give a criterion for semistability of a Higgs bundle having a structure of a system of Hodge bundles. 

\begin{theorem}\label{thm-1}
	Assume that $\deg(\Omega_X^1) \geq 0$. Let $(E, \theta)$ be a Higgs bundle on $X$ which 
	admits a structure of a system of Hodge bundles $E = \bigoplus_{i=0}^{n} E_i$. Suppose that, 
	$\theta\vert_{E_i} : E_i \longrightarrow E_{i-1}\otimes\Omega_X^1$ is an isomorphism of 
	$\mathcal{O}_X$--modules, for all $i \in \{1, \ldots, n\}$. If $E_i$ is semistable, for all 
	$i \in \{1, \ldots, n\}$, then $(E, \theta)$ is a semistable Higgs bundle. 
\end{theorem}

To prove this theorem, we need the following useful inequalities: 
\begin{lemma}[Chebyshev's sum inequalities]\label{chevy}
	Let $\left(a_i\right)_{i=1}^n$ and $\left(b_j\right)_{j=1}^n$ be two finite sequence of real numbers. 
	\begin{enumerate}[(i)]
		\item\label{chev-2} If $a_1 \geq a_2 \geq \cdots \geq a_n$ and $b_1 \leq b_2 \leq \cdots \leq b_n$, 
		then we have 
		\begin{equation}\label{inequ-2}
			n \left(\sum\limits_{i=1}^n a_i b_i \right) \leq \left(\sum\limits_{i=1}^n a_i\right) 
			\left(\sum\limits_{j=1}^n b_j\right)\,. 
		\end{equation} 
		
		\item\label{chev-1} If $a_1 \leq a_2 \leq \cdots \leq a_n$ and $b_1 \leq b_2 \leq \cdots \leq b_n$, 
		then we have 
		\begin{equation}\label{inequ-1}
			\left(\sum\limits_{j=1}^n b_j\right) \left(\sum\limits_{i=1}^n a_i\right) \leq 
			n \left(\sum\limits_{i=1}^n a_i b_i \right)\,. 
		\end{equation} 
	\end{enumerate}
\end{lemma}

\begin{lemma}\label{ineq-lem}
	Let $d \geq 1$ be an integer. Then for any integers $r$ and $n$, with $0 \leq r \leq n$, 
	we have 
	\begin{equation}\label{inequ-3}
	\left(\sum\limits_{i=0}^r i \cdot d^{i-1}\right) \left(\sum\limits_{j=0}^n d^{j}\right) 
	\leq \left(\sum\limits_{i=0}^{n} i \cdot d^{i-1}\right) \left(\sum\limits_{j=0}^r d^{j}\right)\,. 
	\end{equation}
\end{lemma}

\begin{proof}[Proof of Theorem \ref{thm-1}]
	Since $E_i \cong E_0\otimes(\Omega_X^1)^{\otimes i}$, for all $i \in \{0, 1, \ldots, n\}$, 
	we have,
	\begin{equation}\label{deg-E-p}
	\deg(E_i) = i\cdot d^{i-1} \cdot\deg(\Omega_X^1)\cdot \rk(E_0) + d^i \cdot \deg(E_0)\,, 
	\end{equation} 
	and 
	\begin{equation}\label{rk-E-p}
		\rk(E_i) = d^i \cdot \rk(E_0)\,,\,\,\, \forall\, i = 0, \ldots, n\,. 
	\end{equation}
	Now for any integer $k \in \{0, 1, \cdots, n\}$, by \eqref{deg-E-p} and \eqref{rk-E-p} we have, 
	\begin{eqnarray}\label{mu-sum-E-p}
	\mu\left(\bigoplus\limits_{i=0}^k E_i\right) 
	= \frac{\sum\limits_{i=0}^k\deg(E_i)}{\sum\limits_{i=0}^k \rk(E_i)} 
		& = & \frac{\left(\deg(\Omega_X^1) \rk(E_0) \sum\limits_{i=0}^r i \cdot d^{i-1} + 
		\deg(E_0) \sum\limits_{i=0}^k d^i \right)}{\rk(E_0)\sum\limits_{i=0}^k d^i} \nonumber \\ 
		 & = & \frac{\deg(\Omega_X^1) \cdot \sum\limits_{i=0}^r i \cdot d^{i-1}}{\sum\limits_{i=0}^r d^i} 
	\end{eqnarray}
	It follows from \eqref{mu-sum-E-p} and Lemma \ref{ineq-lem} that 
	\begin{equation}\label{mu-E}
		\mu\left(\bigoplus\limits_{i=0}^k E_i\right) \leq \mu(E)\,,\,\,\, \forall\,\, k = 0, \ldots, n\,.  
	\end{equation}
	
	Suppose on the contrary that $(E, \theta)$ is not semistable. 
	Let $F$ be the unique maximal semistable proper Higgs subsheaf of $(E, \theta)$ with 
	\begin{equation}\label{mu-E-F}
		\mu(F) > \mu(E)\,. 
	\end{equation} 
	It follows from \cite[Lemma 2.4]{LSYZ} that $F$ admits a structure of system 
	of Hodge bundle; in particular, $F \cong \bigoplus\limits_{i=0}^{n} F_i$, with 
	$F_i = F\cap E_i$, for all $i = 0, 1, \ldots, n$. 
	
	Since $\theta\vert_{E_i}$ is an isomorphism, we have 
	\begin{equation}\label{F-p}
		F_i \cong \theta(F_i) \subseteq F_{i-1}\otimes \Omega_X^1\,,\,\,\, 
		\forall\,\, i = 0, 1, \ldots, n\,. 
	\end{equation} 
	Therefore, $F_i \neq 0$ implies $F_{i-1} \neq 0$, for all $1 \leq i \leq n$. 
	Let $r \in \{0, \cdots, n\}$ be the largest integer such that $F_r \neq 0$. 
	Then $F = \bigoplus\limits_{i=0}^r F_i$. 
	Now from \eqref{F-p}, we have 
	\begin{equation}\label{rk-F-p}
		0 < \rk(F_r) \leq \rk(F_{r-1}) \leq \cdots \leq \rk(F_0)\,. 
	\end{equation} 
	Since $F_i \neq 0$ and $E_i$ is semistable by assumption, 
	using \eqref{rk-E-p}, we have 
	\begin{equation}\label{Deg-F-p}
		\deg(F_i) \leq \frac{\rk(F_i) \cdot \deg(E_i)/d^i}{\rk(E_0)}\,,\,\forall\, i = 0, 1, \ldots, r\,. 
	\end{equation} 
	Therefore, using \eqref{rk-F-p} and \eqref{deg-E-p}, applying Lemma \ref{chevy} (\ref{chev-2}), 
	from \eqref{Deg-F-p}, we have 
	\begin{equation}\label{eqn2}
	\deg(F) \leq \frac{\sum\limits_{i=0}^r \rk(F_i) \deg(E_i)/d^i}{\rk(E_0)} 
	\leq \frac{\left(\sum\limits_{i=0}^r \rk(F_i)\right)
	\left(\sum\limits_{j=0}^r \deg(E_j)/d^j\right)}{(r+1)\rk(E_0)} 
\end{equation}
Now from \eqref{deg-E-p} and \eqref{eqn2}, applying Lemma \ref{chevy} (\ref{chev-1}), 
we have 
\begin{equation}\label{eqn3}
	\mu(F) \leq \frac{\left(\sum\limits_{j=0}^r \deg(E_j)/d^j\right)}{(r+1)\rk(E_0)}\, 
	\leq \frac{\sum\limits_{i=0}^r \deg(E_i)}{\rk(E_0)\cdot \sum\limits_{i=0}^r d^i} 
	 = \mu\left(\bigoplus\limits_{i=0}^r E_i\right)\,. 
\end{equation} 
Then from \eqref{eqn3} and \eqref{mu-E}, we have 
\begin{equation*}
\mu(F) \leq \mu(E)\,, 
\end{equation*}
	which contradicts \eqref{mu-E-F}. Therefore, $(E, \theta)$ is semistable. 
\end{proof}

\begin{remark}
	Note that, semistability of $E_1 \cong E_0\otimes \Omega_X^1$ 
	forces $\Omega_X^1$ to be semistable. 
	It follows from the relation \eqref{deg-E-p} and \eqref{rk-E-p} that 
	$E = \bigoplus_{i=0}^n E_i$, in Theorem \ref{thm-1}, is semistable if and only 
	if $\deg(\Omega_X^1) = 0$. Therefore, we get many examples of semistable Higgs 
	bundles on $X$ whose underlying vector bundle is not semistable. 
\end{remark} 

\begin{remark}
	If $char(k) > 0$, it is expected that, 
	if $\Omega_X^1$ and all $E_i$ are strongly semistable, then $(E, \theta)$ is strongly semistable; 
	meaning that all the Frobenius pullbacks of $(E, \theta)$ are semistable. 
\end{remark}

We now give an example to show that a semistable Higgs bundle in Theorem \ref{thm-1} may not be stable, 
in general. 

\begin{example}
	Let $X$ be an irreducible smooth complex projective algebraic curve of genus $g \geq 1$. 
	Then $K_X := \Omega_X^1$ is a line bundle of degree $2g -2$ on $X$. 
	Let $Q = (\mathcal{O}_X(1))^{\otimes (g-1)}$ and set $E_1 = Q \oplus Q$. Then $E_1$ is a rank $2$ 
	strictly semistable vector bundle of degree $2g - 2$ on $X$. Take $E_0 := E_1\otimes K_X^{-1}$ and 
	define $E = E_0 \oplus E_1$. Clearly, $\deg(E) = 0$. 
	Fix $\alpha \in \Aut_X(E_1)$ and consider it as an $\mathcal{O}_X$--module isomorphism 
	$\alpha : E_1 \stackrel{\simeq}{\longrightarrow} E_0\otimes K_X$. Define an $\mathcal{O}_X$--module 
	homomorphism $\theta : E \longrightarrow E\otimes K_X$ by the matrix 
	\[
	\theta := 
	\begin{pmatrix}
	0 & \alpha \\ 
	0 & 0 
	\end{pmatrix} \,.
	\]
	Then $(E, \theta)$ is a system of Hodge bundles on $X$ satisfying all conditions in Theorem \ref{thm-1}. 
	So $(E, \theta)$ is a semistable Higgs bundle on $X$. Since $E_1$ is not stable, there is a line 
	subbundle $L_1$ of $E_1$ with $\deg(L_1) = \mu(E_1) = \deg(E_1)/2$. 
	Let $L_0 = \theta(L_1) \otimes K_X^{-1} \subset E_0$ and define $F := L_0 \oplus L_1$. 
	Then $\theta(F) \subset F\otimes K_X$ and 
	$$\deg(F) = \deg(L_0) + \deg(L_1) = 2\cdot\deg(L_1) - \deg(K_X) = \deg(E_1) -(2g - 2) = 0\,.$$ 
	Therefore, $(E, \theta)$ is not stable. 
\end{example}

Unless otherwise mentioned, from now on, we assume that $char(k) = 0$. 
\begin{lemma}\label{lem-3}
	Let $V$ be an unstable torsion free coherent sheaf of $\mathcal{O}_X$--modules on $X$. 
	Let $$0 = V_m \subset V_{m-1} \subset \cdots \subset V_0 = V$$ be the Harder-Narasimhan 
	filtration of $V$. Then for any semistable vector bundle $W$ on $X$, 
	$$0 = V_m \otimes W \subset V_{m-1} \otimes W \subset \cdots \subset V_0 \otimes W = V \otimes W$$ 
	is the Harder-Narasimhan filtration of $V \otimes W$. 
\end{lemma}

\begin{proof}
	For each $i \in \{0, 1, \cdots, m\}$, consider the exact sequence of coherent sheaves : 
	\begin{equation}\label{ex-seq-1}
	0 \longrightarrow V_{i+1} \longrightarrow V_i \longrightarrow V_i/V_{i+1} \longrightarrow 0\,. 
	\end{equation} 
	Since $W$ is locally free, tensoring \eqref{ex-seq-1} with $W$, we get 
	$$(V_i \otimes W)/(V_{i+1} \otimes W) \cong (V_i/V_{i+1}) \otimes W\,,\,\,\, 
	\forall\,\, i = 0, 1, \ldots, m-1 \,.$$ 
	Then the result follows from the fact that $(V_i/V_{i+1}) \otimes W$ is semistable 
	(see e.g., \cite{HL}) and 
	$\mu((V_i/V_{i+1}) \otimes W) = \mu((V_i/V_{i+1})) + \mu(W)$, for all $i = 0, 1, \ldots, m-1$. 
\end{proof}

\begin{lemma}\label{cor-1}
	Let $E$ and $F$ be two isomorphic unstable torsion free coherent sheaf of $\mathcal{O}_X$--modules 
	on $X$. Let $G \subset E$ be the maximal destabilizing subsheaf of $E$. Then for any two 
	$\mathcal{O}_X$--module isomorphisms $f_1 , f_2 : E \longrightarrow F$, we have $f_1(G) = f_2(G)$, 
	and this is the maximal destabilizing subsheaf of $F$. 
\end{lemma}

\begin{proof}
	This follows from the fact that the maximal destabilizing subsheaf 
	is invariant under all $\mathcal{O}_X$--module automorphisms of the coherent sheaf. 
\end{proof}

\begin{theorem}\label{main-thm-1}
	Assume that $\Omega_X^1$ is semistable with $\deg(\Omega_X^1) \geq 0$. 
	Let $(E, \theta)$ be a Higgs bundle on $X$ admitting a structure of a system of Hodge bundles 
	given by $E = \bigoplus\limits_{i=0}^n E_i$ with 
	$\theta\vert_{E_i} : E_i \longrightarrow E_{i-1}\otimes\Omega_X^1$ isomorphisms, for all 
	$i = 1, \ldots, n$. Then $(E, \theta)$ is semistable if and only if $E_0$ is semistable. 
\end{theorem}

\begin{proof}
	Since $E_p \cong E_0 \otimes (\Omega_X^1)^{\otimes p}$, for all $p = 0, 1, \ldots, n$, 
	and $\Omega_X^1$ is semistable, for any $p \in \{0, 1, \cdots, n\}$ we have,
	$E_p$ is semistable if and only if $E_0$ is semistable. Therefore, if $E_0$ is semistable, 
	then $(E, \theta)$ is semistable by Theorem \ref{thm-1}. We now show the converse part. 
	
	Let $(E, \theta)$ be semistable. Tensoring $E$ with a sufficiently large degree line bundle, if 
	required, we may assume that $\deg(E_p) > 0$, for all $p = 0, 1, \ldots, n$. 
	Suppose that, $E_0$ is not semistable. Let $F_p \subset E_p$ 
	be the maximal destabilizing subsheaf of $E_p$, for all $p = 0, 1, \ldots, n$. 
	Since $\theta\vert_{E_p} : E_p \to E_{p-1}\otimes \Omega_X^1$ is an isomorphism, it follows from 
	Lemma \ref{lem-3} and Lemma \ref{cor-1} that $\theta(F_p) = F_{p-1}\otimes \Omega_X^1$, 
	for all $p = 0, 1, \ldots, n$. Therefore, we have 
	\begin{equation}\label{eqn-3}
	\rk(E_p) = d^p \cdot \rk(E_{0})\,\,\,\, \textnormal{and}\,\,\,\, 
	\rk(F_p) = d^p \cdot \rk(F_{0})\,,\,\, \forall\,\,\, p = 0, 1, \ldots, n\,, 
	\end{equation}
	where $d = \rk(\Omega_X^1) = \dim(X)$. 
	Clearly $F = \bigoplus\limits_{p=0}^n F_p$ is a Higgs subsheaf of $(E, \theta)$. 
	Now from \eqref{eqn-3} we have, 
	\begin{equation*}
	\deg(F) = \sum\limits_{p=0}^n \deg(F_p) > 
	\frac{\rk(F_0)}{\rk(E_0)} \sum\limits_{p=0}^n \deg(E_p) = \frac{\rk(F_0)}{\rk(E_0)} \deg(E)\,. 
	\end{equation*}
	Therefore, $\mu(F) > \mu(E)$, which contradicts the fact that $(E, \theta)$ is semistable. 
\end{proof}

\subsection{Criterion for stability of a system of Hodge bundles}
Let $char(k) \geq 0$. 
\begin{definition}
	A Higgs bundle $(E, \theta)$ is said to be \textit{simple} if any non-zero Higgs endomorphism 
	of $(E, \theta)$ is an isomorphism. 
\end{definition}

\begin{proposition}\label{simple}
	Assume that $\deg(\Omega_X^1) > 0$. 
	Let $(E, \theta)$ be a Higgs bundle on $X$ having a structure of a system of Hodge bundles: 
	$E = \bigoplus\limits_{p=0}^n E_p$, with $\theta\vert_{E_p} : E_p \to E_{p-1}\otimes \Omega_X^1$ 
	isomorphism, for all $p = 1, \ldots, n$. If $E_p$ is stable, for all $p = 0, 1, \ldots, n$, then 
	$(E, \theta)$ is simple. 
\end{proposition}

\begin{proof}
	Since $\theta(E_p) \subseteq E_{p-1}\otimes \Omega_X^1$, for all $p = 0, 1, \ldots, n$, the matrix 
	of $\theta$ is strictly block-upper triangular and of the form : 
	\[
	\theta = \begin{pmatrix}
	0 & \theta_{01} & 0 & \cdots & 0 \\ 
	0 & 0 & \theta_{12} & \cdots & 0 \\
	\vdots & \vdots & \vdots &\ddots & \vdots \\ 
	0 & 0 & 0 & \cdots & \theta_{n-1, n} \\
	0 & 0 & 0 & \cdots & 0 
	\end{pmatrix}\,,
	\]
	where $\theta_{ij} \in \mathrm{Iso}_{\mathcal{O}_X}(E_j, E_{i}\otimes\Omega_X^1)$, for all $i, j$. 
	Let $\varphi : E \longrightarrow E$ be a non-zero $\mathcal{O}_X$--module homomorphism with 
	\begin{equation}\label{phi-theta}
		\theta\circ \varphi = (\varphi \otimes \Id_{\Omega_X^1}) \circ \theta\,. 
	\end{equation} 
	For any $i, j \in \{0, 1, \cdots, n\}$, let $\varphi_{ij}$ be the composite homomorphism 
	$$\varphi_{ij} : E_j \hookrightarrow E \stackrel{\varphi}{\longrightarrow} E 
	\stackrel{\pi_i}{\longrightarrow} E_i\,,$$
	where $\pi_i$ is the projection of $E$ onto the $i$-th factor. 
	It follows from \eqref{phi-theta} that the matrix of 
	$\varphi$ is block-upper triangular :  
	\[
	\varphi = \begin{pmatrix}
	\varphi_{00} & \varphi_{01} & \cdots & \varphi_{0n} \\ 
	0 & \varphi_{11} & \cdots & \varphi_{1n} \\
	\vdots & \vdots & \ddots & \vdots \\ 
	0 & 0 & \cdots & \varphi_{nn} 
	\end{pmatrix}
	\]
	Since $\mu(E_i) = \mu(E_0) + i \cdot \mu(\Omega_X^1)$, and $E_i$ are stable by assumption, 
	$\varphi_{ii} = \lambda_{i} \Id_{E_i}$, for some $\lambda_i \in k$, for all 
	$i = 0, 1, \ldots, n$, and $\varphi_{ij} = 0$, for all $i < j$. 
	Therefore, $\varphi$ is the diagonal matrix 
	$\varphi  = \mathrm{diag}(\varphi_{00}, \cdots, \varphi_{nn})$. 
	It follows from the relation \eqref{phi-theta} that 
	$\lambda_0 = \lambda_1 = \cdots = \lambda_n$. 
	Since $\varphi \neq 0$, we must have $\lambda_i \neq 0$, and hence $\varphi$ is an isomorphism. 
\end{proof}

\begin{definition}
	A semistable Higgs bundle is said to be \textit{polystable} if it is a direct sum of stable Higgs bundles. 
\end{definition} 

Every semistable Higgs sheaf $(E, \theta)$ contains a unique maximal polystable Higgs subsheaf, 
called the \textit{socle} of $(E, \theta)$. The socle of $(E, \theta)$ is invariant under all 
Higgs automorphisms of $E$. 

\begin{remark}
	Note that, a simple polystable Higgs bundle is necessarily stable. 
\end{remark}

\begin{theorem}
	Assume that $char(k) \geq 0$, and $\deg(\Omega_X^1) > 0$. 
	Let $(E, \theta)$ be a Higgs bundle on $X$ having a structure of a system of Hodge bundles: 
	$E = \bigoplus\limits_{p=0}^n E_p$, with $\theta\vert_{E_p} : E_p \to E_{p-1}\otimes \Omega_X^1$ 
	isomorphism, for all $p = 1, \ldots, n$. If $E_p$ is stable, for each $p = 0, 1, \ldots, n$, then 
	$(E, \theta)$ is stable. 
\end{theorem}

\begin{proof}
	Clearly $(E, \theta)$ is semistable by Theorem \ref{thm-1}. Suppose that $(E, \theta)$ is not stable. 
	Then its socle $(F, \theta_F) \subset (E, \theta)$ is the unique non-zero proper maximal polystable 
	Higgs subsheaf with $\mu(F) = \mu(E)$. Clearly, $(F, \theta_F)$ is 
	invariant under the $\mathbb{G}_m$--action on $(E, \theta)$. Therefore, $(F, \theta_F)$ admits a 
	structure of a system of Hodge bundles, say $F = \bigoplus\limits_{i=0}^n F_i$, with 
	$\theta_F(F_i) \subseteq F_{i-1}\otimes\Omega_X^1$, for all $i = 0, 1, \ldots, n$. 
	It follows from the proof of \cite[Lemma 2.4]{LSYZ} that, $F_i = F \cap E_i$, for all 
	$i = 0, 1, \ldots, n$. Since $\theta\vert_{E_p}$ is an isomorphism, we have 
	$F_p \cong \theta(F_p) \subseteq F_{p-1} \otimes \Omega_X^1$, for all $p = 1, \ldots, n$. 
	Let $r \in \{0, 1, \cdots, n\}$ be the largest integer such that $F_r \neq 0$. 
	Then $F = \bigoplus\limits_{p=0}^r F_p$. 
	Since $F$ is a proper subsheaf of $E$, there is \textbf{at least one $p \in \{0, 1, \cdots, r\}$} 
	such that $F_p \neq E_p$. Since all $E_p$ are stable, we have, 
	\begin{equation}\label{deg-F-p}
		\deg(F_p) \leq \frac{\rk(F_p)\cdot \deg(E_p)/d^p}{\rk(E_0)}\,,\,\,\, \forall\,\, p = 0, 1, \ldots, r\,, 
	\end{equation}
	and the \textbf{inequality \eqref{deg-F-p} is strict for at least one $p \in \{0, 1, \cdots, r\}$}. 
	Then from \eqref{deg-F-p}, following the inequality computations as in proof of Theorem \ref{thm-1}, 
	we have 
	$$\mu(F) < \mu\left(\bigoplus\limits_{p=0}^r E_p \right) \leq \mu(E)\,,$$
	which contradicts the fact that $\mu(F) = \mu(E)$. Therefore, $(E, \theta)$ is polystable. 
	Then by Proposition \ref{simple}, $(E, \theta)$ is stable. 
\end{proof}

\begin{theorem}\label{main-thm-2}
	Assume that $char(k) = 0$. 
	Let $X$ be an irreducible smooth projective algebraic curve of genus $g \geq 2$. 
	Let $(E, \theta)$ be a Higgs bundle on $X$ admitting a structure of a system of 
	Hodge bundles : $E = \bigoplus\limits_{p=0}^n E_p$ with 
	$\theta\vert_{E_p} : E_p \longrightarrow E_{p-1}\otimes\Omega_X^1$ an isomorphism, for all 
	$p = 1, \ldots, n$. Then $(E, \theta)$ is stable if and only if $E_p$ is stable, 
	for all $p = 0, 1, \ldots, n$. 
\end{theorem}

\begin{proof}
	Suppose that, $(E, \theta)$ is stable. It follows from Theorem \ref{main-thm-1} that $E_p$ is 
	semistable, for all $p = 0, 1, \ldots, n$. Since $K_X := \Omega_X^1$ is a line bundle on $X$, 
	for any $p \in \{0, 1, \cdots, n\}$ we see that, $E_p$ stable if and only if $E_0$ is stable. 
	Suppose that $E_0$ is not stable. 
	Then there is a non-zero proper stable subsheaf $G_0 \subset E_0$ with $\mu(G_0) = \mu(E_0)$. 
	Since $\theta^p$ is an isomorphism of $E_p$ onto $E_0\otimes K_X^{\otimes p}$, 
	for all $p$, there is a subsheaf $G_p \subset E_p$ such that 
	$\theta^p : G_p \longrightarrow G_0 \otimes K_X^{\otimes p}$ is isomorphism, 
	for all $p = 0, 1, \ldots, n$. Then $G = \bigoplus\limits_{p=0}^n G_p$ is a Higgs subsheaf of 
	$(E, \theta)$. Now a similar computation as in the proof of Theorem \ref{main-thm-1} 
	shows that $\mu(G) = \mu(E)$, which contradicts the fact that $(E, \theta)$ is stable. 
\end{proof}

\begin{remark}
	It is expected that for $\dim_k(X) \geq 2$ with $\Omega_X^1$ is stable and $\deg(\Omega_X^1) > 0$, 
	if $(E, \theta)$ in Theorem \ref{main-thm-2} is stable then all $E_p$ are polystable. 
\end{remark}

\begin{remark}
	Note that, in the proofs of all Theorems in this Section, we have used only semistability (or stability) 
	of $\Omega_X^1$ and the condition $\deg(\Omega_X^1) \geq 0\,\,(\, > 0)$. Therefore, with appropriate 
	notion of semistability and stability of pairs $(E, \theta)$ with $\theta \in H^0(X, \sEnd(E) \otimes V)$, 
	all Theorems in this Section \ref{sec-3} hold if we replace $\Omega_X^1$ with any semistable (or stable) 
	vector bundle $V$ on $X$ of degree $\geq 0$ (or, $>0$). 
\end{remark}

\subsection{Examples of unstable system of Hodge bundles} 
We now give two examples to show that the isomorphism conditions in Theorem \ref{thm-1} are crucial. 

\begin{example}
	Let $X$ be a smooth complex projective curve of genus $g \geq 2$. 
	Let $L_0$ be a line bundle of degree $d > 2g - 2$ on $X$. 
	Let $E_1$ be a non-trivial extension of $\mathcal{O}_X$ and $L_0\otimes K_X$. 
	So we have a short exact sequence of $\mathcal{O}_X$--modules 
	$$0 \longrightarrow \mathcal{O}_X \longrightarrow E_1 
	\stackrel{\theta_1}{\longrightarrow} L_0\otimes K_X \longrightarrow 0\,.$$ 
	Then $E_1$ is semistable, but not necessarily stable. 
	Let $E = E_0 \bigoplus E_1$, where $E_0 = L_0$. 
	Then $\deg(E) = \deg(E_0) + \deg(E_1) = 2d + 2g - 2$, and hence 
	$\mu(E) = 2(d + g -1)/3$. Since $d > 2g -2$, we have $\mu(E) < d = \mu(L_0)$. 
	Define a Higgs field $\theta \in H^0(X, \sEnd(E)\otimes K_X)$ by 
	\[
	\theta = \begin{pmatrix}
	0 & \theta_1 \\
	0 & 0
	\end{pmatrix}\,.
	\]
	Then $(E, \theta)$ is a Higgs bundle having a structure of a system of Hodge bundles on $X$. 
	Note that, \textcolor{black}{\textbf{$\theta_1$ is surjective, but not isomorphism}}. 
	Since $L_0$ is a $\theta$--invariant, $(E, \theta)$ is not semistable.
\end{example}

\begin{example}
	Let $X$ be a smooth projective curve of genus $g \geq 2$ over $k$. Let $L_0$ be a line 
	bundle on $X$ of positive degree. Let $L_1 = L_0^{\vee}$ and $E = L_0 \bigoplus L_1$. Since 
	$\deg(\sHom(L_1, L_0\otimes K_X)) = 2\cdot \deg(L_0) + (2g-2)$, choosing $L_0$ with $\deg(L_0)$ 
	sufficiently large, we can find a non-zero $\mathcal{O}_X$--module homomorphism 
	$\theta_1 : L_1 \longrightarrow L_0\otimes K_X$. 
	Note that, \textcolor{black}{\textbf{$\theta_1$ is injective}}, because both 
	$L_1$ and $L_0\otimes K_X$ are line bundles, but \textcolor{black}{\textbf{$\theta_1$ is not 
	an isomorphism}}. Define an $\mathcal{O}_X$--module homomorphism 
	$\theta : E \longrightarrow E\otimes K_X$ by 
	\[
	\theta := \begin{pmatrix}
	0 & \theta_1 \\ 
	0 & 0 
	\end{pmatrix}\,.
	\]
	Then $(E, \theta)$ is a Higgs bundle having a structure of a system of Hodge bundles on $X$. 
	Since $L_0$ is a $\theta$--invariant subbundle of positive degree, $(E, \theta)$ is not semistable. 
\end{example}

We now shows that, if all $E_p$ are not semistable $(E, \theta)$ may fail to be semistable. 

\begin{example}
	Let $X$ be a smooth complex projective curve of genus $g \geq 2$. Fix a square root $K^{1/2}$ of 
	the cotangent bundle $K_X$ on $X$. Let $L_0$ be a positive degree line bundle on $X$. 
	Consider the line bundles $L_1 = \left(L_0^{\vee}\otimes K_X^{-1/2}\right)^{\otimes 2}$ and 
	$L_2 = L_0 \otimes K_X$ on $X$. Let $E = E_0 \bigoplus E_1$, where $E_0 = L_0$ and 
	$E_1 = L_1 \bigoplus L_2$. Clearly, $\deg(E) = 0$. 
	Consider the $\mathcal{O}_X$--module homomorphism 
	$\theta : E \longrightarrow E \otimes K_X$ defined by 
	\[
	\theta = \begin{pmatrix}
	0 & \theta_1 \\ 
	0 & 0 
	\end{pmatrix}\,, 
	\]
	where $\theta_1 : E_1 = L_1\bigoplus (L_0 \otimes K_X) \longrightarrow E_0 \otimes K_X$ is the projection 
	homomorphism onto the second factor. Then $(E, \theta)$ is a Higgs bundle of degree $0$ having a 
	structure of a system of Hodge bundles on $X$. Note that, 
	\textcolor{black}{\textbf{$E_1$ is not semistable}}, and 
	$\theta\vert_{E_1} : E_1 \to E_0\otimes K_X$ is surjective, but not isomorphism. 
	Since $E_0$ is a $\theta$--invariant subbundle of positive degree, 
	$(E, \theta)$ is not semistable. 
\end{example}

\section{Generalized Oper} 
\subsection{Semistability of generalized oper}
It is not know if the Higgs bundle $(\gr(\mathcal{F}^{\bullet}(E)), \theta_{\nabla})$, defined in 
Section \ref{as-higgs-bun} associated to a Griffiths transversal filtration $\mathcal{F}^{\bullet}(E)$ 
with respect to a flat connection $\nabla$ on $E$, is semistable or not. In this section, we give a 
criterion for semistability of $(\gr(\mathcal{F}^{\bullet}(E)), \theta_{\nabla})$. 

\begin{definition}\label{oper}
	An \textit{oper} is a triple $(E, \mathcal{F}^{\bullet}(E), \nabla)$ consists of a   vector bundle 
	$E$ on $X$ together with a   flat connection $\nabla$ and a Griffiths transversal filtration 
	$\mathcal{F}^{\bullet}(E)$ (see \eqref{filtration}) of $E$ by its   subbundles 
	(see Definition \ref{G-transv}), such that the quotients $\mathcal{F}^i(E)/\mathcal{F}^{i+1}(E)$ 
	are line bundles on $X$ and the $\mathcal{O}_X$--linear homomorphisms $\theta_{\nabla}^i$ in 
	\eqref{theta-i} induced by $\nabla$ are isomorphisms, for all $i = 1, \ldots, n-1$. 
\end{definition}

Let us first recall the following well-known result. 
\begin{proposition}\cite[p.~186]{Si}\label{prop-1}
	Let $X$ be a connected smooth complex projective curve of genus $g \geq 1$. 
	Let $(E, \mathcal{F}^{\bullet}(E), \nabla)$ be an oper on $X$. 
	Then the associated Higgs bundle $(\gr(\mathcal{F}^{\bullet}(E)), \theta_{\nabla})$ 
	on $X$ is semistable. 
\end{proposition}

We give a natural generalization of the above result for higher ranks and higher dimensional 
algebraic varieties. 

\begin{definition}\label{gen-oper}
	A \textit{generalized oper} is a triple $(E, \mathcal{F}^{\bullet}(E), \nabla)$ consists of a   
	vector bundle $E$ on $X$ together with a   flat connection $\nabla$ and a Griffiths transversal 
	filtration $\mathcal{F}^{\bullet}(E)$ (see \eqref{filtration}) of $E$ by its   subbundles 
	(see Definition \ref{G-transv}), such that the quotients 
	$\gr^i(\mathcal{F}^{\bullet}(E)) := \mathcal{F}^i(E)/\mathcal{F}^{i+1}(E)$ 
	are semistable   vector bundles on $X$ and the $\mathcal{O}_X$--linear homomorphisms 
	$\theta_{\nabla}^i$ in \eqref{theta-i} induced by $\nabla$ are isomorphisms, 
	for all $i = 1, \ldots, n-1$. 
\end{definition}

\begin{remark}
	Note that, if $(E, \mathcal{F}^{\bullet}(E), \nabla)$ is a generalized oper on $X$ and 
	$\deg(\Omega_X^1) > 0$, then $\mathcal{F}^{\bullet}(E)$ is the Harder-Narasimhan filtration of $E$. 
\end{remark}

\begin{theorem}\label{thm-2}
	Let $(E, \mathcal{F}^{\bullet}(E), \nabla)$ be a generalized oper on $X$. 
	If $\deg(\Omega_X^1) \geq 0$, then the associated Higgs bundle 
	$(\gr(\mathcal{F}^{\bullet}(E)), \theta_{\nabla})$ on $X$ is semistable, 
	where $\gr(\mathcal{F}^{\bullet}(E)) = \bigoplus\limits_{i=0}^{n-1} \gr^i(\mathcal{F}^{\bullet}(E))$. 
\end{theorem}

\begin{proof}
	Since $(\gr(\mathcal{F}^{\bullet}(E)), \theta_{\nabla})$ has a structure of a system of Hodge bundles, 
	all $\gr^i(\mathcal{F}^{\bullet}(E))$ are semistable and all 
	$\theta_{\nabla}^i$ are isomorphisms, by Theorem \ref{thm-1}, 
	$(\gr(\mathcal{F}^{\bullet}(E)), \theta_{\nabla})$ is semistable. 
\end{proof}

\begin{remark}
	\begin{enumerate}
		\item If $X$ is a smooth complex projective curve of genus $g \geq 1$, 
		then we get Proposition \ref{prop-1} as a corollary to the Theorem \ref{thm-2}. 
		\item With appropriate notion of logarithmic Higgs semistability, Theorem \ref{thm-2} 
		holds for logarithmic connections singular over an effective divisor, using similar techniques. 
	\end{enumerate}
\end{remark}

\subsection{Semistability of Connections}

As before, let $X$ be a smooth polarized projective variety over and algebraically closed field $k$ of 
characteristic $p \geq 0$, and the cotangent bundle $\Omega_X^1$ is semistable and of non-negative degree. 

\begin{definition}
	Let $E$ be a torsion free coherent sheaf on $X$ together with a connection $\nabla : E \to E \otimes \Omega_X^1$. 
	Then the pair $(E, \nabla)$ is said to be \textit{semistable} (respectively, \textit{stable}) if for 
	any non-zero proper $\mathcal{O}_X$--submodule $F \subset E$ with torsion free quotient 
	sheaf $E/F$ on $X$ such that $\nabla(F) \subseteq F \otimes \Omega_X^1$, we have $\mu(F) \leq \mu(E)$ 
	(respectively, $\mu(F) < \mu(E)$). 
\end{definition}

\begin{definition}
	A \textit{partial oper} is a triple $(E, \mathcal{F}^{\bullet}(E), \nabla)$ consisting of a vector 
	bundle $E$ on $X$ together with a flat connection $\nabla$ and a filtration $\mathcal{F}^{\bullet}(E)$ of $E$ 
	by its subbundles on $X$ which is Griffiths transversal with respect to $\nabla$ such that the induced Higgs 
	bundle $(\gr(\mathcal{F}^{\bullet}(E)), \theta_{\nabla})$ is semistable on $X$. 
\end{definition}

Given a vector bundle $E$ on $X$ together with a semistable flat connection $\nabla$ on $E$, Lan-Sheng-Yang-Zuo 
proved that there is a filtration $\mathcal{F}^{\bullet}(E)$ of $E$ such that the triple 
$(E, \mathcal{F}^{\bullet}(E), \nabla)$ is a partial oper on $X$ (see \cite[Theorem 2.2]{LSYZ}). 
We now prove some sort of converse of the above result. 

\begin{theorem}\label{thm-3}
	Let $E$ be a vector bundle on a smooth projective variety $X$ over an algebraically closed field $k$ of positive 
	characteristic. Let $\nabla : E \longrightarrow E \otimes \Omega_X^1$ be a connection (not necessarily flat) on $E$. 
	Let 
	$$\mathcal{F}^{\bullet}(E) : 0 = \mathcal{F}^n(E) \subsetneq \mathcal{F}^{n-1}(E) \subsetneq \cdots 
	\subsetneq \mathcal{F}^1(E) \subsetneq \mathcal{F}^0(E) = E$$ 
	be a $\nabla$-Griffiths transversal filtration of $E$ by its subbundles such that the induced $\mathcal{O}_X$--module 
	homomorphism 
	$\theta_{\nabla} : \gr(\mathcal{F}^{\bullet}(E)) \longrightarrow \gr(\mathcal{F}^{\bullet}(E)) \otimes \Omega_X^1$ 
	is a Higgs field on $\gr(\mathcal{F}^{\bullet}(E))$ (i.e., $\theta_{\nabla} \bigwedge \theta_{\nabla} = 0$ in 
	$H^0(X, \sEnd(\gr(\mathcal{F}^{\bullet}(E)))\otimes\Omega_X^2)$), and the Higgs bundle 
	$(\gr(\mathcal{F}^{\bullet}(E)), \theta_{\nabla})$ is semistable. Then the pair $(E, \nabla)$ is semistable. 
\end{theorem}

\begin{proof}
	If $(E, \nabla)$ were not semistable, then there is a non-zero proper $\mathcal{O}_X$--submodule 
	$F \subset E$ such that $\nabla(F) \subseteq F \otimes \Omega_X^1$ and 
	\begin{equation}\label{eqn-5.1}
		\mu(F) > \mu(E)\,.
	\end{equation}
	The filtration $\mathcal{F}^{\bullet}(E)$ induces a filtration on $F$ 
	$$\mathcal{F}^{\bullet}(F)\,\,\, :\,\,\, 0 = \mathcal{F}^n(F) \subseteq \mathcal{F}^{n-1}(F) \subseteq 
	\cdots \subseteq \mathcal{F}^1(F) \subseteq \mathcal{F}^0(F) = F\,,$$ 
	where $\mathcal{F}^i(F) = \mathcal{F}^i(E) \cap F$, for all $i$. 
	Since $\nabla(F) \subseteq F \otimes \Omega_X^1$, the injective homomorphisms 
	$\iota_i : \mathcal{F}^i(F)/\mathcal{F}^{i+1}(F) \hookrightarrow \mathcal{F}^i(E)/\mathcal{F}^{i+1}(E)$, 
	induced by the inclusions $\mathcal{F}^i(F) \subseteq \mathcal{F}^i(E)$, fits into the following commutative 
	diagram of $\mathcal{O}_X$--module homomorphisms 
	$$
	\xymatrix{
	\mathcal{F}^i(F)/\mathcal{F}^{i+1}(F) \ar@{^(->}[r]^{\iota_i} \ar[d]_{\theta_i'} & 
	\mathcal{F}^i(E)/\mathcal{F}^{i+1}(E) \ar[d]^{\theta_i} \\ 
	\left(\mathcal{F}^{i-1}(F)/\mathcal{F}^{i}(F)\right) \otimes \Omega_X^1 \ar@{^(->}[r]^{\iota_{i+1}} & 
	\left(\mathcal{F}^{i-1}(E)/\mathcal{F}^{i}(E)\right)\otimes \Omega_X^1 
	}
	$$
	where $\theta_i'$ is the restriction of $\theta_i$, for all $i$. 
	So $(\gr(\mathcal{F}^{\bullet}(F)), \theta_{\nabla}')$ is a non-zero Higgs subsheaf of the semistable Higgs bundle 
	$(\gr(\mathcal{F}^{\bullet}(E)), \theta_{\nabla})$. Now a simple degree rank computation shows that 
	$$\mu(F) = \mu(\gr(\mathcal{F}^{\bullet}(F))) \leq \mu(\gr(\mathcal{F}^{\bullet}(E))) = \mu(E)\,,$$ 
	which contradicts the inequality \eqref{eqn-5.1}. Therefore, $(E, \nabla)$ is semistable. 
\end{proof}

\begin{corollary}\label{cor-2}
	Let $E$ be a vector bundle on $X$ together with a flat connection $\nabla : E \to E \otimes \Omega_X^1$. 
	Suppose that $E$ admits a filtration by its subbundles 
	$$\mathcal{F}^{\bullet}(E)\,\,\, :\,\,\, 0 = \mathcal{F}^n(E) \subsetneq \mathcal{F}^{n-1}(E) \subsetneq 
	\cdots \subsetneq \mathcal{F}^1(E) \subsetneq \mathcal{F}^0(E) = E$$ 
	which is Griffiths transversal with respect to $\nabla$, and the $\nabla$--induced $\mathcal{O}_X$--module 
	homomorphisms $\theta_i : \mathcal{F}^i(E)/\mathcal{F}^{i+1}(E) \longrightarrow 
	\left(\mathcal{F}^{i-1}(E)/\mathcal{F}^i(E)\right)\otimes \Omega_X^1$ 
	are isomorphisms, for all $i = 1, \ldots, n-1$. Then the pair $(E, \nabla)$ is semistable if 
	$\mathcal{F}^i(E)/\mathcal{F}^{i+1}(E)$ is semistable, for all $i = 0, 1, \ldots, n-1$. 
\end{corollary}

\begin{proof}
	If $p = char(k)$ is zero, then for any flat connection $\nabla$ on 
	$E$, the pair $(E, \nabla)$ is automatically semistable. This follows from the fact that any non-zero coherent 
	sheaf $F$ on $X$ admitting a flat connection has zero first Chern class, and hence has zero slope with respect 
	to any polarization on $X$. This is not the case if $p > 0$. 

	If $char(k) = p > 0$, since $(\gr(\mathcal{F}^{\bullet}(E)), \theta_{\nabla})$ is semistable by Theorem \ref{thm-1}, 
	we are done by using Theorem \ref{thm-3}. 
\end{proof}

\begin{remark}
	Corollary \ref{cor-2} was proved over smooth projective curve in positive characteristics 
	by Joshi-Pauly (see \cite[Proposition 3.4.4]{JP}). 
\end{remark}

\section*{Acknowledgement}
This work is supported by the Post-doctoral fellowship of the Institute of Mathematical Sciences (HBNI), Chennai, India.



\begin{thebibliography}{AAAA}
	\bibitem[Bi]{Bi} Indranil Biswas, Criterion for connections on principal bundles over a 
	pointed Riemann surface, \href{https://doi.org/10.1515/coma-2017-0010}{\textit{Complex Manifolds} 
	\textbf{4} (2017), 155--171}. 
	
	\bibitem[JP]{JP} Kirti Joshi and Christian Pauly, Hitchin-Mochizuki morphism, opers and Frobenius-destabilized 
	vector bundles over curves, \href{http://dx.doi.org/10.1016/j.aim.2015.01.004}{\textit{Adv. Math.} 
	\textbf{274} (2015) 39--75}. 
	
	\bibitem[HL]{HL} Daniel Huybrechts, Manfred Lehn, \textit{The geometry of moduli spaces of sheaves}. 
	Second edition. Cambridge Mathematical Library. \textit{Cambridge University Press, Cambridge}, 
	2010. xviii+325 pp. 
	
	\bibitem[LSYZ]{LSYZ} Guitang Lan, Mao Sheng, Yanhong Yang and Kang Zuo, Semistable Higgs bundles 
	of small ranks are strongly Higgs semistable, \href{https://arxiv.org/abs/1311.2405}{arXiv:1311.2405}. 
	
	\bibitem[Si1]{Si} Carlos Simpson, \href{http://www.ams.org/books/conm/522/10300}{Iterated 
	destabilizing modifications for vector bundles with connection}, 
	\textit{Vector bundles and complex geometry}, 183--206, \textit{Contemp. Math.}, \textbf{522}, 
	\textit{Amer. Math. Soc., Providence, RI}, 2010. 
	
	\bibitem[Si2]{Si2} Carlos T. Simpson, Higgs bundles and local systems, 
	\href{http://www.numdam.org/item?id=PMIHES_1992__75__5_0}{\textit{Inst. Hautes \'{E}tudes 
	Sci. Publ. Math.} \textbf{75} (1992), 5--95.} 
\end{thebibliography}
\end{document}